\documentclass{birkjour}

\usepackage{graphicx}

 \newtheorem{thm}{Theorem}[section]

 \theoremstyle{definition}
 
 \theoremstyle{remark}

 \numberwithin{equation}{section}

\begin{document}
\title[Dual-complex k-Pell quaternions]{Dual-complex k-Pell quaternions}  

\author[F\"{u}gen Torunbalc{\i} Ayd{\i}n]{F\"{u}gen Torunbalc{\i}  Ayd{\i}n}
\address{%
Yildiz Technical University\\
Faculty of Chemical and Metallurgical Engineering\\
Department of Mathematical Engineering\\
Davutpasa Campus, 34220\\
Esenler, Istanbul,  TURKEY}

\email{faydin@yildiz.edu.tr ; ftorunay@gmail.com}

\thanks{*Corresponding Author}


\begin{abstract}

In this paper, dual-complex k-Pell numbers and dual-complex k-Pell quaternions are defined. Also, some algebraic properties of dual-complex k-Pell numbers and quaternions which are connected with dual-complex numbers and k-Pell numbers are investigated. Furthermore, the Honsberger identity, the d'Ocagne's identity, Binet's formula, Cassini's identity, Catalan's identity for these quaternions are given. 
\end{abstract}

\maketitle

\section{Introduction}

In 1971,  Horadam studied  on the Pell and Pell-Lucas sequences and  he gave Cassini-like formula as follows \cite{A}:   
\begin{equation}\label{E1}
{P}_{n+1}{P}_{n-1}-{P}_{n}^2=(-1)^n
\end{equation}
and Pell identities
\begin{equation}\label{E2}
\left\{\begin{array}{l}
P_{r}\,P_{n+1}+P_{r-1}\,P_{n}=P_{n+r},\cr
P_{n}(P_{n+1}+P_{n-1})=P_{2n},\cr
P_{2n+1}+P_{2n}=2\,P_{n+1}^2-2\,P_{n}^2-(-1)^n,\cr
P_{n}^2+P_{n+1}^2=P_{2n+1},\cr
P_{n}^2+P_{n+3}^2=5(P_{n+1}^2+P_{n+2}^2),\cr
P_{n+a}\,P_{n+b}-P_{n}\,P_{n+a+b}=(-1)^n\,P_{n}\,P_{n+a+b},\cr
P_{-n}=(-1)^{n+1}\,P_{n}.
\end{array}\right.
\end{equation}
and in 1985, Horadam and Mahon obtained Cassini-like formula as follows \cite{B}
\begin{equation}\label{E3}
q_{n+1}\,q_{n-1}-q_{n}^2=8\,(-1)^{n+1}.
\end{equation} 
Many kinds of generalizations of the Pell sequence have been presented in the literature \cite{C}-\cite{E}.
Furthermore, Torunbalc{\i} Ayd{\i}n, K\"{o}kl\"{u} introduced the generalizations of the Pell sequence in 2017 \cite{E} as follows:
\begin{equation}\label{E4}
\left\{\begin{array}{rl}
\mathbb{P}_{0}=&{q},\,\mathbb{P}_{1}=p,\,{\mathbb{P}_{2}}=2p+q, \,\, {p}\,{q}  \in\mathbb{Z} \cr
{\mathbb{P}_{n}}=&2{\mathbb{P}_{n-1}}+\,{\mathbb{P}_{n-2}},\,\ n\geq 2 \cr
or \cr
{\mathbb{P}_n}=&(p-2q)\,P_{n}+q\,P_{n+1}=p\,P_{n}+q\,P_{n-1}\cr
\end{array}\right.
\end{equation} 
In 2013, the k-Pell sequence $\{P_{k,n}\}_{n\in\mathbb{N}}$ is defined by Catarino and Vasco \cite{F} as follows
\begin{equation}\label{E5}
\left\{\begin{array}{rl}
{P}_{k,0}=&0,\,\,{P}_{k,1}=1 \cr
{P}_{k,n+1}=&2\,{P}_{k,n}+k\,{P}_{k,n-1},\,\ n\geq 1 \cr
or \cr
\{{P}_{k,n}\}_{n\in\mathbb{N}}=&\{\,0,\,1,\,2,\,k+4,\,4\,k+8,\,\,k^2+12\,k+16,...\}
\end{array}\right.
\end{equation} 
Here, ${k}$  is a positive real number.
The studies that follows is based on the work of Catarino and Vasco \cite{G}-\cite{J}.
First the idea to consider Pell quaternions it was suggested by Horadam in paper \cite{K}. In the literature the reader can find Pell quaternions and studies on their properties in \cite{L}-\cite{S}. 

\par In 2017, Catarino and Vasco introduced dual k-Pell quaternions and Octonions \cite{R} as follows:
\begin{equation}\label{E6}
\widetilde{{R}_{k,n}}=\widetilde{{P}_{k,n}}\,{e_0}+\widetilde{{P}_{k,n+1}}\,{e_1}+\widetilde{{P}_{k,n+2}}\,{e_2}+\widetilde{{P}_{k,n+3}}\,{e_3},
\end{equation} 
where
${\widetilde{{P_{k,n}}}=P_{k,n}+\varepsilon\,P_{k,n+1}}$,\,\, ${P_{k,n}=2\,P_{k,n-1}+k\,P_{k,n-2}}$,\,\, \,${n\ge2}$
\begin{equation*}
{e_0}=1,\,{e_i}^{2}=-1,\,e_i\,e_j=-e_j\,e_i,\,\,i,\,j=1,\,2,\,3,
\end{equation*}
\begin{equation*}
\varepsilon\,\ne0,\,\,0\,\varepsilon=\varepsilon\,0=0,\,\,1\,\varepsilon=\varepsilon\,1=\varepsilon,\,\,\varepsilon^2=0.
\end{equation*}
\par In 2018, G\"{u}l introduced k-Pell quaternions and k-Pell-Lucas quaternions \cite{S} as follows:
\begin{equation}\label{E7}
{Q_P}_{k,n}={P}_{k,n}+i\,{P}_{k,n+1}+j\,{P}_{k,n+2}+k\,{P}_{k,n+3}
\end{equation}
and
\begin{equation}\label{E8}
{Q_{PL}}_{k,n}={p}_{k,n}+i\,{p}_{k,n+1}+j\,{p}_{k,n+2}+k\,{p}_{k,n+3}
\end{equation} 
where ${\,i,\,j,\,k\,}$ satisfy the multiplication rules
\begin{equation*}
{i}^{2}={j}^{2}={k}^{2}=i\,j\,k=-1\,,\ \ i\,j=k=-j\,i\,,\quad j\,k=i=-k\,j\,,\quad k\ i=j=-i\ k\,.
\end{equation*} 
\par In 2018, Torunbalc{\i} Ayd{\i}n introduced dual-complex Pell and Pell-Lucas quaternions \cite{T} ({submitted}) as follows: 
\begin{equation}\label{E9}
\begin{aligned}
\mathbb{DC}^{P_n}=\{{Q_P}_{n}={P}_{n}+i\,{P}_{n+1}+\varepsilon\,{P}_{n+2}+i\,\varepsilon\,{P}_{n+3}\,\left.\right|\,\, {P}_{n}\,,\, n-th,\,\text{Pell number}\} 
\end{aligned}
\end{equation}
and
\begin{equation}\label{E10}
\begin{aligned}
\mathbb{DC}^{PL_n}=\{{{Q}_{PL}}_{n}=&{Q}_{n}+i\,{Q}_{n+1}+\varepsilon\,{Q}_{n+2}+i\,\varepsilon\,{Q}_{n+3}\,\left.\right|\,\,\, {Q}_{n}\,,\cr
& n-th,\,\text{Pell-Lucas number}\} 
\end{aligned}
\end{equation}
where
\begin{equation*}
\begin{aligned}
{i}^2=-1,\,\varepsilon\neq 0,\, \,{\varepsilon}^{2}=0,\,\,\, (i\,\varepsilon)^2=0.
\end{aligned}
\end{equation*}
\par Majernik has introduced the multi-component number system \cite{U}. There are three types of the four-component number systems which have been constructed by joining the complex, binary and dual two-component numbers. Later, Farid Messelmi has defined the algebraic properties of the dual-complex numbers in the light of this study \cite{V}. There are many applications for  the theory of dual-complex numbers. In 2017, G\"{u}ng\"{o}r and Azak have defined dual-complex Fibonacci and dual-complex Lucas numbers and their properties \cite{W}.\\
Dual-complex numbers \cite{V} $w$ can be expressed in the form as 
\begin{equation}\label{E11}
\begin{aligned}
\mathbb{DC}=\{\, w={z}_{1}+\varepsilon {z}_{2} \, |\, {z}_{1},\,{z}_{2}\,\in\mathbb{C} \,\, \text{where} \, \, \, \varepsilon^2=0,\, \varepsilon \neq 0 \}.
\end{aligned}
\end{equation}
Here if ${z_1=x_1+i\,x_2}$ and ${z_2=y_1+i\,y_2}$, then any dual-complex number can be written 
\begin{equation}\label{E12}
\begin{aligned}
w=x_1+ix_2+\varepsilon\,y_1+i\,\varepsilon \,y_2
\end{aligned}
\end{equation}
\begin{equation*}
i^2=-1,\,\,\varepsilon\neq 0,\,\,\varepsilon^2=0,\,\,\,(i\,\varepsilon)^2=0.  
\end{equation*}
Addition, substraction and multiplication of any two dual-complex numbers $w_1$ and $w_2$ are defined by
\begin{equation}\label{E13}
\begin{aligned}
w_1\pm w_2=&(z_1+\varepsilon z_2)\pm (z_3+\varepsilon z_4)=(z_1\pm z_3)+\varepsilon(z_2\pm z_4), \cr
w_1\times\,w_2=&(z_1+\varepsilon z_2)\times\,(z_3+\varepsilon z_4)=z_1\,z_3+\varepsilon\,(z_2\,z_4+z_2\,z_3).
\end{aligned}
\end{equation}
On the other hand, the division of two dual-complex numbers are given by
\begin{equation}\label{E14}
\begin{aligned}
&\frac{w_1}{w_2}=\frac{z_1+\varepsilon z_2}{z_3+\varepsilon z_4} \cr
&\frac{(z_1+\varepsilon z_2)(z_3-\varepsilon z_4)}{(z_3+\varepsilon z_4)(z_3-\varepsilon z_4)}=
\frac{z_1}{z_3}+\varepsilon \,\frac{z_2\,z_3-z_1\,z_4}{z_3^2}.
\end{aligned}
\end{equation}
If ${Re(w_2)\neq0}$,then the division ${\frac{w_1}{w_2}}$ is possible. The dual-complex numbers are defined by the basis $\{1,i,\varepsilon,i\,\varepsilon \}$. Therefore, dual-complex numbers, just like quaternions, are a generalization of complex numbers by means of entities specified by four-component numbers. But real and dual quaternions are non commutative, whereas, dual-complex numbers are commutative. The real and dual quaternions form a division algebra, but dual-complex numbers form a commutative ring with characteristics $0$. Moreover, the multiplication of these numbers gives the dual-complex numbers the structure of 2-dimensional complex Clifford Algebra and 4-dimensional real Clifford Algebra.  
The base elements of the dual-complex numbers satisfy the following commutative multiplication scheme (Table 1).
\begin{table}[]
\centering
\caption{Multiplication scheme of dual-complex numbers}
\begin{tabular}{lllll}
\hline
 $x$&  $1$&  $i$& $\varepsilon$& $i\,\varepsilon$\cr  
\hline
 $1$&  $1$&  $i$&  $\varepsilon$& $i\,\varepsilon$\cr  
 $i$&  $i$&  $-1$&  $i\,\varepsilon$& $-\varepsilon$\cr 
$\varepsilon$&  $\varepsilon$&  $i\,\varepsilon$&  $0$& $0$\cr
$i\,\varepsilon$& $i\,\varepsilon$& $-\varepsilon$& $0$& $0$\cr 
\hline 
\end{tabular}
\end{table}
\\
Five different conjugations can operate on dual-complex numbers \cite{V} as follows: 
\begin{equation}\label{E15}
\begin{aligned}
w=&x_1+ix_2+\varepsilon\,y_1+i\,\varepsilon y_2 \cr
{w}^{*_1}=&(x_1-ix_2)+\varepsilon (y_1-i\,y_2)=(z_1)^*+\varepsilon\,(z_2)^*,\cr
{w}^{*_2}=&(x_1+i\,x_2)-\varepsilon \,(y_1+i\,y_2)=z_1-\varepsilon\,z_2,\cr
{w}^{*_3}=&(x_1-i\,x_2)-\varepsilon \,(y_1-i\,y_2)=z_1^*-\varepsilon\,z_2^*, \cr
{w}^{*_4}=&(x_1-i\,x_2)(1-\varepsilon \,\frac{y_1+i\,y_2 }{x_1+ix_2} \,)=(z_1)^*(1-\varepsilon\frac{z_2}{z_1}), \cr
{w}^{*_5}=&(y_1+i\,y_2)-\varepsilon (x_1+i\,x_2)=z_2-\varepsilon z_1. 
\end{aligned}
\end{equation}
Therefore, the norm of the dual-complex numbers is defined as 
\begin{equation}\label{E16}
\begin{aligned}
{N}_{w}^{*_1}=&\left\| {w\times{w}^{*_1}} \right\|=\sqrt{\left|{z}_{1}\right|^2+2\,\varepsilon Re({z}_{1}\,{z_2}^*)}, \cr 
{N}_{w}^{*_2}=&\left\| {w\times{w}^{*_2}} \right\|=\sqrt{{z}_{1}^2}, \cr
{N}_{w}^{*_3}=&\left\| {w\times{w}^{*_3}} \right\|=\sqrt{\left|{z}_{1}\right|^2-2\,i\,\varepsilon Im({z}_{1}\,{z_2}^*)},\cr
{N}_{w}^{*_4}=&\left\| {w\times{w}^{*_4}} \right\|=\sqrt{\left|{z}_{1}\right|^2},\cr
{N}_{w}^{*_5}=&\left\| {w\times{w}^{*_5}} \right\|=\sqrt{{z}_{1}\,{z}_{2}+\varepsilon ({z}_{2}^2-{z_1}^2)}. 
\end{aligned}
\end{equation} 
\par In this paper, the dual-complex k-Pell numbers and the dual-complex k-Pell quaternions will be defined. The aim of this work is to present in a unified manner a variety of algebraic properties of the dual-complex k-Pell quaternions as well as both the k-Pell numbers and the dual-complex numbers. In particular, using five types of conjugations, all the properties established for dual-complex numbers and k-Pell numbers are also given for the dual-complex k-Pell quaternions. In addition, the Honsberger identity, the d'Ocagne's identity, Binet's formula, Cassini's identity, Catalan's identity for these quaternions are given.
\section{The dual-complex k-Pell numbers} 
In this section, the dual-complex k-Pell, k-Pell-Lucas and modified k-Pell numbers can be defined by the basis $\{1,\,i,\,\varepsilon,\,i\,\varepsilon\,\}$, where $i$,\,\,\,$\varepsilon$ \,and\, $i\,\varepsilon$ satisfy the conditions 
\begin{equation*}
i^2=-1,\,\,\varepsilon\neq 0,\,\,\varepsilon^2=0,\,\,\,(i\,\varepsilon)^2=0.  
\end{equation*}
as follows
\begin{equation}\label{F1}
\begin{aligned}
\mathbb{DC}P_{k,n}=&({P}_{k,n}+i\,{P}_{k,n+1})+\varepsilon \,({P}_{k,n+2}+i\,{P}_{k,n+3}) \cr
=& {P}_{k,n}+i\,{P}_{k,n+1}+\varepsilon \,{P}_{k,n+2}+i\,\varepsilon\,{P}_{k,n+3},
\end{aligned}
\end{equation}
\begin{equation}\label{F2}
\begin{aligned}
\mathbb{DC}PL_{k,n}=&({PL}_{k,n}+i\,{PL}_{k,n+1})+\varepsilon \,({PL}_{k,n+2}+i\,{PL}_{k,n+3}) \cr
=& {PL}_{k,n}+i\,{PL}_{k,n+1}+\varepsilon \,{PL}_{k,n+2}+i\,\varepsilon\,{PL}_{k,n+3},
\end{aligned}
\end{equation} 
and
\begin{equation}\label{F3}
\begin{aligned}
\mathbb{DC}{MP}_{k,n}=&({MP}_{k,n}+i\,{MP}_{k,n+1})+\varepsilon \,({MP}_{k,n+2}+i\,{MP}_{k,n+3}) \cr
=& {MP}_{k,n}+i\,{MP}_{k,n+1}+\varepsilon \,{MP}_{k,n+2}+i\,\varepsilon\,{MP}_{k,n+3}.
\end{aligned}
\end{equation} 
With the addition, substraction and multiplication by real scalars of two dual-complex k-Pell numbers, the dual-complex k-Pell number can be obtained again. 
Then, the addition and subtraction of the dual-complex k-Pell numbers are defined by 
\begin{equation}\label{F4}
\begin{array}{rl}
\mathbb{DC}P_{k,n}\pm \mathbb{DC}P_{k,m}=&({P}_{k,n}\pm{P}_{k,m})+i\,({P}_{k,n+1}\pm{P}_{k,m+1}) \cr
&+\varepsilon \,({P}_{k,n+2}\pm{P}_{k,m+2})+i\,\varepsilon\,({P}_{k,n+3}\pm{P}_{k,m+3}) 
\end{array}
\end{equation}
The multiplication of a dual-complex k-Pell number by the real scalar $\lambda$ is defined as 
\begin{equation}\label{F5}
{\lambda}\,\mathbb{DC}P_{k,n}=\lambda\,P_{k,n}+i\,\lambda\,{P}_{k,n+1}+\varepsilon\,\lambda\,{P}_{k,n+2}+i\,\varepsilon\,\lambda\,\,{P}_{k,n+3}.
\end{equation} 
By using (Table 1) the multiplication of two dual-complex k-Pell numbers is defined by
\begin{equation}\label{F6}
{\begin{array}{rl}
\mathbb{DC}P_{k,n}\times\,\mathbb{DC}P_{k,m}=&({P}_{k,n}\,{P}_{k,m}-{P}_{k,n+1}\,{P}_{k,m+1}) \cr
&+i\,({P}_{k,n+1}\,{P}_{k,m}+{P}_{k,n}\,{P}_{k,m+1}) \cr
&+\varepsilon \,({P}_{k,n}\,{P}_{k,m+2}-{P}_{k,n+1}\,{P}_{k,m+3} \cr
&+{P}_{k,n+2}\,{P}_{k,m}-{P}_{k,n+3}\,{P}_{k,m+1}) \\
&+i\,\varepsilon\,({P}_{k,n+1}\,{P}_{k,m+2}+{P}_{k,n}\,{P}_{k,m+3} \cr
&+{P}_{k,n+3}\,{P}_{k,m}+{P}_{k,n+2}\,{P}_{k,m+1}) \cr
=&\mathbb{DC}P_{k,m}\times\,\mathbb{DC}P_{k,n}\,.
\end{array}}
\end{equation} 
Also, there exits five conjugations as follows:
\begin{equation}\label{F7}
\begin{aligned}
\mathbb{DC}{P}_{k,n}^{*_1}={P}_{k,n}-i\,{P}_{k,n+1}+\varepsilon\,{P}_{k,n+2}-i\,\varepsilon\,{P}_{k,n+3},\, \text{complex-conjugation} \cr
\end{aligned}
\end{equation}
\begin{equation}\label{F8}
\begin{aligned}
\mathbb{DC}{P}_{k,n}^{*_2}={P}_{k,n}+i\,{P}_{k,n+1}-\varepsilon\,{P}_{k,n+2}-i\,\varepsilon\,{P}_{k,n+3},\, \text{dual-conjugation} \cr
\end{aligned}
\end{equation}
\begin{equation}\label{F9}
\begin{aligned}
\mathbb{DC}{P}_{k,n}^{*_3}={P}_{k,n}-i\,{P}_{k,n+1}-\varepsilon\,{P}_{k,n+2}+i\,\varepsilon\,{P}_{k,n+3},\ \text{coupled-conjugation} \cr
\end{aligned}
\end{equation}
\begin{equation}\label{F10}
\begin{aligned}
\mathbb{DC}{P}_{k,n}^{*_4}=&({P}_{k,n}-i\,{P}_{k,n+1})\,(\,1-\varepsilon\,\frac{{P}_{k,n+2}+i\,{P}_{k,n+3}}{{P}_{k,n}+i\,{P}_{k,n+1}}\,), \cr
&\text{dual-complex-conjugation} \cr
\end{aligned}
\end{equation}
\begin{equation}\label{F11}
\begin{aligned}
\mathbb{DC}{P}_{k,n}^{*_5}={P}_{k,n+2}+i\,{P}_{k,n+3}-\varepsilon\,{P}_{k,n}-i\,\varepsilon\,{P}_{k,n+1},\, \text{anti-dual-conjugation}.
\end{aligned}
\end{equation}
In this case, we can give the following relations:
\begin{equation}\label{F12}
\begin{aligned}
\mathbb{DC}{P}_{k,n}\,(\mathbb{DC}{P}_{k,n})^{*_1}=&P_{k,n}^2+P_{k,n+1}^2+2\,\varepsilon \,({P}_{k,n}\,P_{k,n+2}+P_{k,n+1}\,P_{k,n+3})\cr
=&P_{k,n}^2+P_{k,n+1}^2+2\,\varepsilon\,P_{k,2n+3}, 
\end{aligned}
\end{equation}
\begin{equation}\label{F13}
\mathbb{DC}{P}_{k,n}\,(\mathbb{DC}{P}_{k,n})^{*_2}=P_{k,n}^2-P_{k,n+1}^2+2\,i\,({P}_{k,n}\,P_{k,n+1}),
\end{equation}
\begin{equation}\label{F14}
\begin{aligned}
\mathbb{DC}{P}_{k,n}\,(\mathbb{DC}{P}_{k,n})^{*_3}=&P_{k,n}^2+P_{k,n+1}^2+2\,i\,\varepsilon\,({P}_{k,n}\,P_{k,n+3}-P_{k,n+1}\,P_{k,n+2})\cr
=&P_{k,n}^2+P_{k,n+1}^2-4\,i\,\varepsilon\,(-1)^n\,k^n,
\end{aligned} 
\end{equation}
\begin{equation}\label{F15}
\mathbb{DC}{P}_{k,n}\,(\mathbb{DC}{P}_{k,n})^{*_4}=P_{k,n}^2+P_{k,n+1}^2, 
\end{equation}
\begin{equation}\label{F16}
\mathbb{DC}{P}_{k,n}+\,(\mathbb{DC}{P}_{k,n})^{*_1}=2\,(P_{k,n}+\,\varepsilon{P}_{k,n+2}), 
\end{equation}
\begin{equation}\label{F17}
\mathbb{DC}{P}_{k,n}+\,(\mathbb{DC}{P}_{k,n})^{*_2}=2\,(P_{k,n}+i\,{P}_{k,n+1}), 
\end{equation}
\begin{equation}\label{F18}
\mathbb{DC}{P}_{k,n}+\,(\mathbb{DC}{P}_{k,n})^{*_3}=2\,(P_{k,n}+i\,\varepsilon{P}_{k,n+3}), 
\end{equation}
\begin{equation}\label{F19}
\begin{aligned}
({P}_{k,n}+i\,{P}_{k,n+1})\,(\mathbb{DC}{P}_{k,n})^{*_4}=&(P_{k,n}^2+{P}_{k,n+1}^2-\varepsilon{P}_{k,2n+3}+2\,i\,\varepsilon (-1)^n\,k^n) \cr
=&({P}_{k,n}-i\,{P}_{k,n+1})\,(\mathbb{DC}{P}_{k,n})^{*_2}, 
\end{aligned}
\end{equation}
\begin{equation}\label{F20}
\varepsilon\,\mathbb{DC}{P}_{k,n}+\,(\mathbb{DC}{P}_{k,n})^{*_5}=P_{k,n+2}+i\,{P}_{k,n+3}, 
\end{equation}
\begin{equation}\label{F21}
\begin{aligned}
\mathbb{DC}{P}_{k,n}-\varepsilon\,(\mathbb{DC}{P}_{k,n})^{*_5}=P_{k,n}+i\,{P}_{k,n+1}.
\end{aligned}
\end{equation}
The norm of the dual-complex k-Pell numbers ${\,\mathbb{DC}{P}_{k,n}}$ is defined in five different ways as follows
\begin{equation}\label{F22}
{\begin{array}{rl}
{N}_{\mathbb{DC}{P}_{k,n}^{*_1}}=&\|\mathbb{DC}{P}_{k,n}\times\,(\mathbb{DC}{P}_{k,n})^{*_1}\|^2 \cr
=&({P}_{k,n}^2+{{P}_{k,n+1}^2})+2\,\varepsilon(\,{P}_{k,n}\,{P}_{k,n+2}+{P}_{k,n+1}\,{P}_{k,n+3})\, \cr
=&{P}_{k,n}^2+{{P}_{k,n+1}^2}+2\,\varepsilon{P}_{k,2n+3}, 
\end{array}} 
\end{equation}
\begin{equation}\label{F23}
{\begin{array}{rl}
{N}_{\mathbb{DC}{P}_{k,n}^{*_2}}=&\|\mathbb{DC}{P}_{k,n}\times\,(\mathbb{DC}{P}_{k,n})^{*_2}\|^2 \cr
=&|({P}_{k,n}^2-{P}_{k,n+1}^2)+2\,i\,{P}_{k,n}\,{P}_{k,n+1}\,|,
\end{array}} 
\end{equation}
\begin{equation}\label{F24}
{\begin{array}{rl}
{N}_{\mathbb{DC}{P}_{k,n}^{*_3}}=&\|\mathbb{DC}{P}_{k,n}\times\,(\mathbb{DC}{P}_{k,n})^{*_3}\|^2 \cr
=&({P}_{k,n}^2+{P}_{k,n+1}^2)+2\,i\,\varepsilon({P}_{k,n}\,{P}_{k,n+3}-{P}_{k,n+1}\,{P}_{k,n+2})\, \cr
=&|{P}_{k,n}^2+{P}_{k,n+1}^2-4\,i\,\varepsilon (-1)^{n}\,k^n|,
\end{array}}  
\end{equation}
\begin{equation}\label{F25}
{\begin{array}{rl}
{N}_{\mathbb{DC}{P}_{k,n}^{*_4}}=&\|\mathbb{DC}{P}_{k,n}\times\,(\mathbb{DC}{P}_{k,n})^{*_4}\|^2 \cr
=&{P}_{k,n}^2+{P}_{k,n+1}^2\,.
\end{array}}  
\end{equation}
\begin{thm} Let ${\mathbb{DC}{P}_{k,n}}$, ${\mathbb{DC}{PL}_{k,n}}$ and ${\mathbb{DC}{MP}_{k,n}}$   be the dual-complex k-Pell number, the dual-complex k-Pell-Lucas number and the dual-complex modified k-Pell number respectively. Then, the following relations hold 
\begin{equation}\label{F26}
\mathbb{DC}{P}_{k,n+2}=2\,\mathbb{DC}{P}_{k,n+1}+k\,\mathbb{DC}{P}_{k,n},
\end{equation} 
\begin{equation}\label{F27}
\mathbb{DC}{PL}_{k,n+2}=2\,\mathbb{DC}{PL}_{k,n+1}+k\,\mathbb{DC}{PL}_{k,n},
\end{equation} 
\begin{equation}\label{F28}
\mathbb{DC}{MP}_{k,n}=\mathbb{DC}{P}_{k,n}+k\,\mathbb{DC}{P}_{k,n-1},
\end{equation} 
\begin{equation}\label{F29}
\mathbb{DC}{MP}_{k,n}=\mathbb{DC}{P}_{k,n+1}-\mathbb{DC}{P}_{k,n},
\end{equation} 
\begin{equation}\label{F30}
\mathbb{DC}{PL}_{k,n}=2\,(\mathbb{DC}{P}_{k,n+1}-\mathbb{DC}{P}_{k,n}),
\end{equation} 
\begin{equation}\label{F31}
\mathbb{DC}{PL}_{k,n+1}=2\,(\mathbb{DC}{P}_{k,n+1}+\mathbb{DC}{P}_{k,n})\,.
\end{equation} 
\end{thm}
\begin{proof} 
Proof of equalities can easily be done.
\end{proof}  
\section{The dual-complex k-Pell and k-Pell-Lucas quaternions}
In this section, firstly the dual-complex k-Pell quaternions will be defined. The dual-complex k-Pell quaternions and the dual-complex k-Pell-Lucas quaternions and the dual-complex modified k-Pell quaternions are defined by using the dual-complex Pell numbers and the dual-complex Pell-Lucas numbers respectively, as follows 
\begin{equation}\label{G1}
\begin{aligned}
\mathbb{DC}^{P_{k,n}}=\{{Q_{P}}_{k,n}=&{P}_{k,n}+i\,{P}_{k,n+1}+\varepsilon\,{P}_{k,n+2}+i\,\varepsilon\,{P}_{k,n+3}\,\left.\right|\, {P}_{k,n}\,,\cr
& n-th,\,\text{k-Pell number}\} 
\end{aligned}
\end{equation}
\begin{equation}\label{G2}
\begin{aligned}
\mathbb{DC}^{PL_{k,n}}=\{{Q_{PL}}_{k,n}=&{PL}_{k,n}+i\,{PL}_{k,n+1}+\varepsilon\,{PL}_{k,n+2}+i\,\varepsilon\,{PL}_{k,n+3}\,\left.\right|\, \cr
&{PL}_{k,n}\,, n-th,\,\text{k-Pell-Lucas number}\} 
\end{aligned}
\end{equation}
and
\begin{equation}\label{G3}
\begin{aligned}
\mathbb{DC}^{MP}_{k,n}=\{{Q_{MP}}_{k,n}=&{MP}_{k,n}+i\,{MP}_{k,n+1}+\varepsilon\,{MP}_{k,n+2}+i\,\varepsilon\,{MP}_{k,n+3}\,\left.\right|\,\cr
&{MP}_{k,n}\,,n-th,\,\text{modified k-Pell number}\} 
\end{aligned}
\end{equation}
where
\begin{equation*}
{i}^2=-1,\,\varepsilon\neq 0,\, \,{\varepsilon}^{2}=0,\,\,\, (i\,\varepsilon)^2=0.
\end{equation*}
Let ${Q_P}_{k,n}$ and ${Q_P}_{k,m}$ be two dual-complex k-Pell quaternions such that
\begin{equation}\label{G4} 
{Q_P}_{k,n}={P}_{k,n}+i\,{P}_{k,n+1}+\varepsilon\,{P}_{k,n+2}+i\,\varepsilon\,{P}_{k,n+3}
\end{equation}
and 
\begin{equation}\label{G5}
{Q_P}_{k,m}={P}_{k,m}+i\,{P}_{k,m+1}+\varepsilon\,{P}_{k,m+2}+i\,\varepsilon\,{P}_{k,m+3}
\end{equation}
Then, the addition and subtraction of two dual-complex k-Pell quaternions are defined in the obvious way,  
\begin{equation}\label{G6}
{\begin{array}{rl}
{Q_P}_{k,n}\pm{\,{Q_P}_{k,m}}=&({P}_{k,n}+i\,{P}_{k,n+1}+\varepsilon\,{P}_{k,n+2}+i\,\varepsilon\,{P}_{k,n+3}) \cr 
&\pm ({P}_{k,m}+i\,{P}_{k,m+1}+\varepsilon\,{P}_{k,m+2}+i\,\varepsilon\,{P}_{k,m+3}) \cr
=&({{P}_{k,n}}\pm{{P}_{k,m}})+i\,({P}_{k,n+1}\pm{P}_{k,m+1}) \cr
&+\varepsilon\,({P}_{k,n+2}\pm{P}_{k,m+2})+i\,\varepsilon\,({P}_{k,n+3}\pm{P}_{k,m+3}).
\end{array}}
\end{equation}
Multiplication of two dual-complex k-Pell quaternions is defined by 
\begin{equation}\label{G7}
{\begin{array}{rl}
{Q_P}_{k,n}\times\,{Q_P}_{k,m}=&({P}_{k,n}+i\,{P}_{k,n+1}+\varepsilon\,{P}_{k,n+2}+i\,\varepsilon\,{P}_{k,n+3}) \cr 
&({P}_{k,m}+i\,{P}_{k,m+1}+\varepsilon\,{P}_{k,m+2}+i\,\varepsilon\,{P}_{k,m+3}) \cr
=&({P}_{k,n}{P}_{k,m}-{P}_{k,n+1}{P}_{k,m+1})\cr 
&+i\,({P}_{k,n+1}{P}_{k,m}+{P}_{k,n}{P}_{k,m+1})\cr 
&+\varepsilon\,({P}_{k,n}{P}_{k,m+2}-{P}_{k,n+1}{P}_{k,m+3} \cr
& \quad+{P}_{k,n+2}{P}_{k,m}-{P}_{k,n+3}{P}_{k,m+1}) \cr
&+i\,\varepsilon\,({P}_{k,n+1}{P}_{k,m+2}+{P}_{k,n}{P}_{k,m+3} \cr
& \quad+{P}_{k,n+3}{P}_{k,m}+{P}_{k,n+2}{P}_{k,m+1}) \cr  
=&{Q_P}_{k,m}\times\,{Q_P}_{k,n}\,.
\end{array}}
\end{equation}
The scaler and the dual-complex vector parts of the dual-complex k-Pell quaternion $({Q_P}_{k,n})$ are denoted by 
\begin{equation}\label{G8}
{S}_{Q_{P_{k,n}}}={P}_{k,n} \ \ \text{and} \ \ \ {V}_{Q_{P_{k,n}}}=i\,{P}_{k,n+1}+\varepsilon\,{P}_{k,n+2}+i\,\varepsilon\,{P}_{k,n+3}.	
\end{equation}
Thus, the dual-complex k-Pell quaternion ${Q_P}_{k,n}$  is given by 
\begin{equation*}
{Q_{P_{k,n}}}={S}_{Q_{P_{k,n}}}+{V}_{Q_{P_{k,n}}}\,.	
\end{equation*}   
\par The five types of conjugation given for the dual-complex k-Pell numbers are the same within the dual-complex k-Pell quaternions. Furthermore, the conjugation properties for these quaternions are given by the relations in (\ref{F7})-(\ref{F11}). \,
In the following theorem, some properties related to the dual-complex k-Pell quaternions are given. 
\begin{thm} Let ${Q_P}_{k,n}$  be the dual-complex k-Pell quaternion. In this case, we can give the following relations:
\begin{equation}\label{G9}
2\,{Q_P}_{k,n+1}+k\,{{Q}_P}_{k,n}={Q_P}_{k,n+2},
\end{equation}  
\begin{equation}\label{G10}
\begin{array}{rl}
({Q_P}_{k,n+1})^2+k\,({Q_P}_{k,n})^2=& {Q_P}_{k,2n+1}-{P}_{k,2n+3}+i\,{P}_{k,2n+2} \cr
&+\varepsilon\,({P}_{k,2n+3}-2\,{P}_{k,2n+5}) \cr
&+3\,i\,\varepsilon\,{P}_{k,2n+4},
\end{array}
\end{equation} 
\begin{equation}\label{G11}
{\begin{array}{rl}
({Q_P}_{k,n+1})^2-k^2\,({Q_P}_{k,n-1})^2=& 2\,{Q_P}_{k,2n}-2\,({P}_{k,2n+2}-i\,{P}_{k,2n+1}+\varepsilon\,{P}_{k,2n+4} \cr
&-3\,i\,\varepsilon\,{P}_{k,2n+3}),
\end{array}}
\end{equation} 
\begin{equation}\label{G12}
{Q_P}_{k,n}-i\,({Q_P}_{k,n+1})^{*_3}-\varepsilon\,{Q_P}_{k,n+2}-i\,\varepsilon\,{Q_P}_{k,n+3}={P}_{k,n}-{P}_{k,n+2}+2\,\varepsilon\,{P}_{k,n+4}.
\end{equation}
\end{thm}
\begin{proof}
(\ref{G9}):By using (\ref{G4}) we get,
\begin{equation*}
{\begin{array}{rl}
2\,{Q_P}_{k,n+1}+k\,{Q_P}_{k,n}=&(2\,{P}_{k,n+1}+k\,{P}_{k,n})+i\,(2\,{P}_{k,n+2}+k\,{P}_{k,n+1}) \cr
&+\varepsilon\,(2\,{P}_{k,n+3}+k\,{P}_{k,n+2}) \cr
&+i\,\varepsilon(2\,{P}_{k,n+4}+k\,{P}_{k,n+3}) \cr
=&{P}_{k,n+2}+i\,{P}_{k,n+3}+\varepsilon\,{P}_{k,n+4}+i\,\varepsilon{P}_{k,n+5} \cr
=&{Q_P}_{k,n+2}\,.
\end{array}}
\end{equation*}
(\ref{G10}): By using (\ref{G4}) we get,
\begin{equation*}
{\begin{array}{rl}
({Q_P}_{k,n+1})^2+k\,({Q_P}_{k,n})^2=& ({P}_{k,n+1}^2+k\,{P}_{k,n}^2)-({P}_{k,n+2}^2+k\,{P}_{k,n+1}^2) \cr
&+2\,i\,({P}_{k,n+1}\,{P}_{k,n+2}+k\,{P}_{k,n}\,{P}_{k,n+1}) \cr
&+2\,\varepsilon\,[({P}_{k,n+1}\,{P}_{k,n+3}+k\,{P}_{k,n}\,{P}_{k,n+2}) \cr
& \quad \quad -({P}_{k,n+2}\,{P}_{k,n+4}+k\,{P}_{k,n+1}\,{P}_{k,n+3})] \cr
&+2\,i\,\varepsilon\,[({P}_{k,n+1}\,{P}_{k,n+4}+k\,{P}_{k,n}\,{P}_{k,n+3}) \cr
& \quad \quad +({P}_{k,n+2}\,{P}_{k,n+3}+k\,{P}_{k,n+1}\,{P}_{k,n+2})]\cr
=&({P}_{k,2n+1}-{P}_{k,2n+3})+2\,i\,{P}_{k,2n+2} \cr 
&+2\,\varepsilon\,({P}_{k,2n+3}-{P}_{k,2n+5})+2\,i\,\varepsilon\,(2\,{P}_{k,2n+4}) \cr
=&{{Q}_P}_{k,2n+1}-{P}_{k,2n+3}+i\,{P}_{k,2n+2} \cr
&+\varepsilon\,({P}_{k,2n+3}-2\,{P}_{k,2n+5})+3\,i\,\varepsilon\,(P_{k,2n+4}). 
\end{array}}
\end{equation*}
(\ref{G11}): By using (\ref{G4}) we get,
\begin{equation*}
{\begin{array}{rl}
({{Q}_P}_{k,n+1})^2-k^2\,({{Q}_P}_{k,n-1})^2=&2\,({P}_{k,2n}-2\,{P}_{k,2n+2})+2\,i\,(2\,{P}_{k,2n+1}) \cr
&+2\,\varepsilon\,({P}_{k,2n+2}-{P}_{k,2n+4})+2\,i\,\varepsilon\,(4\,{P}_{k,2n+3}) \cr 
=& 2\,({P}_{k,2n}+i\,{P}_{k,2n+1}+\varepsilon\,{P}_{k,2n+2}+i\,\varepsilon\,{P}_{k,2n+3}) \cr
&-2\,{P}_{k,2n+2}+2\,i\,{P}_{k,2n+1}-2\,\varepsilon\,{P}_{k,2n+4} \cr
&+6\,i\,\varepsilon\,{P}_{k,2n+3}\cr 
=& 2\,{{Q}_P}_{k,2n}-2\,({P}_{k,2n+2}-i\,{P}_{k,2n+1}+\varepsilon\,{P}_{k,2n+4} \cr
&-3\,i\,\varepsilon\,{P}_{k,2n+3}). 
\end{array}}
\end{equation*}
(\ref{G12}): By using (\ref{G4}) and (\ref{F9}) we get,
\begin{equation*}
{\begin{array}{rl}
{Q_P}_{k,n}-i\,{Q_P}_{k,n+1}^{*_3}-\varepsilon\,{Q_P}_{k,n+2}-i\,\varepsilon\,{Q_P}_{k,n+3}=&({P}_{k,n}-{P}_{k,n+2})+2\,\varepsilon\,{P}_{k,n+4}.
\end{array}}
\end{equation*}
\end{proof} 
\begin{thm} For \,$n,m\ge0$ the Honsberger identity for  the dual-complex k-Pell quaternions ${Q_P}_{k,n}$ and ${Q_P}_{k,m}$ \, is given by
\begin{equation}\label{G13}
\begin{array}{rl}
k\,{Q_P}_{k,n-1}\,{Q_P}_{k,m}+{Q_P}_{k,n}\,{Q_P}_{k,m+1}=& {Q_P}_{k,n+m}-{P}_{k,n+m+2}+i\,{P}_{k,n+m+1} \cr
&+\varepsilon\,({P}_{k,n+m+2}-2\,{P}_{k,n+m+4}) \cr
&+3\,i\,\varepsilon\,{P}_{k,n+m+3}.
\end{array}
\end{equation}
\end{thm}
\begin{proof}
(\ref{G13}): By using (\ref{G4}) we get,
\begin{equation*}
{\begin{array}{rl}
k\,{Q_P}_{k,n-1}\,{Q_P}_{k,m}+{Q_P}_{k,n}\,{Q_P}_{k,m+1}=&(k\,{P}_{k,n-1}{P}_{k,m}+{P}_{k,n}{P}_{k,m+1}) \cr
& \quad-(k\,{P}_{k,n}{P}_{k,m+1}+{P}_{k,n+1}{P}_{k,m+2})\cr
+&\,i\,[(k\,{P}_{k,n-1}{P}_{k,m+1}+{P}_{k,n}{P}_{k,m+2}) \cr& \quad+(k\,{P}_{k,n}{P}_{k,m}+{P}_{k,n+1}{P}_{k,m+1})] \cr
+&\varepsilon\,[(k\,{P}_{k,n-1}{P}_{k,m+2}+{P}_{k,n}{P}_{k,m+3}) \cr
& \quad-(k\,{P}_{k,n}{P}_{k,m+3}+{P}_{k,n+1}{P}_{k,m+4}) \cr
& \quad+(k\,{P}_{k,n+1}{P}_{k,m}+{P}_{k,n+2}{P}_{k,m+1}) \cr
& \quad-(k\,{P}_{k,n+2}{P}_{k,m+1}+{P}_{k,n+3}{P}_{k,m+2})] \cr
+&i\,\varepsilon\,[(k\,{P}_{k,n-1}{P}_{k,m+3}+{P}_{k,n}{P}_{k,m+4}) \cr
& \quad+(k\,{P}_{k,n}{P}_{k,m+2}+{P}_{k,n+1}{P}_{k,m+3}) \cr
& \quad+(k\,{P}_{k,n+1}{P}_{k,m+1}+{P}_{k,n+2}{P}_{k,m+2}) \cr
& \quad+(k\,{P}_{k,n+2}{P}_{k,m}+{P}_{k,n+3}{P}_{k,m+1})] \cr
=&({P}_{k,n+m}-{P}_{k,n+m+2})+2\,i\,{P}_{k,n+m+1} \cr
&+2\,\varepsilon\,({P}_{k,n+m+2}-{P}_{k,n+m+4}) \cr
&+4\,i\,\varepsilon\,{P}_{k,n+m+3} \cr
=&{Q_P}_{k,n+m}-{P}_{k,n+m+2}+i\,{P}_{k,n+m+1} \cr
&+\varepsilon\,({P}_{k,n+m+2}-2\,{P}_{k,n+m+4}) \cr
&+3\,i\,\varepsilon\,{P}_{k,n+m+3}\,.
\end{array}}
\end{equation*}
where the identity $k\,{P}_{k,n-1}{P}_{k,m}+{P}_{k,n}{P}_{k,m+1}={P}_{k,n+m}$ is used \,\cite{F}. 
\end{proof}	
\begin{thm} Let ${Q_P}_{n}$ be the dual-complex k-Pell quaternion.Then, sum formula for these quaternions is as follows:
\begin{equation}\label{G14}
\sum\limits_{s=0}^{n}{{Q_P}_{k,s}}=\frac{1}{k+1}[\,{{Q}_P}_{k,n+1}+k\,{{Q}_P}_{k,n}-{{Q}_P}_{k,1}+{{Q}_P}_{k,0}\,] .
\end{equation}
\end{thm}
\begin{proof}
(\ref{G14}): Since $\sum\nolimits_{i=0}^{n}{P}_{k,i}=\frac{1}{k+1}(-1+{P}_{k,n+1}+k\,{P}_{k,n})$, \cite{F}, \, we get
\begin{equation*}
\begin{aligned}
  \sum\limits_{s=0}^{n}{{Q_P}_{k,s}}=&\sum\limits_{s=0}^{n}{{P}_{k,s}}+i\,\sum\limits_{s=0}^{n}{{P}_{k,s+1}}+\varepsilon\,\sum\limits_{s=0}^{n}{{P}_{k,s+2}}+i\,\varepsilon\,\sum\limits_{s=0}^{n}{{P}_{k,s+3}} \cr
=&\frac{1}{k+1}\,[\,(-1+P_{k,n+1}+k\,P_{k,n})+i\,(-1+{P}_{k,n+2}+k\,{P_{k,n+1}})\cr
&\quad +\varepsilon\,(-2-k+{P}_{k,n+3}+k\,{P_{k,n+2}}) \cr
&\quad +i\,\varepsilon\,(-4-3\,k+{P}_{k,n+4}}+k\,{P_{k,n+3})\,] \cr	
=&\frac{1}{k+1}\,[\,(-{P}_{k,1}+{P}_{k,0}+P_{k,n+1}+k\,P_{k,n}) \cr
&\quad +i\,(-{P}_{k,2}+{P}_{k,1}+{P}_{k,n+2}+k\,{P_{k,n+1}})\cr
&\quad +\varepsilon\,(-P_{k,3}+P_{k,2}+P_{k,n+3}+k\,P_{k,n+2}) \cr
&\quad +i\,\varepsilon\,(-P_{k,4}+P_{k,3}+{P}_{k,n+4}}+k\,{P_{k,n+3})\,] \cr
=&\frac{1}{k+1}\,[\,{Q_P}_{k,n+1}+k\,{Q_P}_{k,n}-{Q_P}_{k,1}+{Q_P}_{k,0}\,] \,. 
\end{aligned}
\end{equation*}
\end{proof}
\begin{thm} \textbf{Binet's Formula} 
Let ${Q_P}_{n}$ be the dual-complex k-Pell quaternion. For $n\ge 1$, Binet's formula for these quaternions is as follows:
\begin{equation}\label{G15}
{Q_P}_{k,n}=\frac{1}{\alpha -\beta }\left( \,\hat{\alpha}\,\,{\alpha}^{n}-\hat{\beta}\,{\beta}^{n} \right)\,
\end{equation}
where
\begin{equation*}
\begin{array}{l}
\hat{\alpha}=1+i\,{\alpha}+\varepsilon\,{\alpha}^2+i\,\varepsilon\,{\alpha}^3,\,\,\,\,\, \alpha={1+\sqrt2}
\end{array}
\end{equation*}
and
\begin{equation*}
\begin{array}{l}
\hat{\beta}=1+i\,{\beta}+\varepsilon\,{\beta}^2+i\,\varepsilon\,{\beta}^3,\,\,\,\,\, \beta={1-\sqrt2}.
\end{array}
\end{equation*}
\end{thm}
\begin{proof} 
Binet's formula of k-Pell number \cite{F} is 
\begin{equation}\label{G16}
{P}_{k,n}=\frac{1}{\alpha -\beta}\left({\alpha}^{n}-{\beta}^{n} \right)\,.
\end{equation}
where \, \, $\alpha=1+\sqrt{1+k}\,\,\,\,,\beta=1-\sqrt{1+k}\,,\,\,\,\alpha +\beta =2\,,\, \, \, \alpha -\beta =2\,\sqrt{1+k}\,,\, \, \alpha \beta =-k\,$. \\
Binet's formula of k-Pell quaternion \cite{T} is 
\begin{equation*}
{QP}_{k,n}=\frac{1}{\alpha -\beta }\left(\hat{\alpha}\,{\alpha }^{n}-\hat{\beta}\,{\beta }^{n} \right)\,.
\end{equation*}
where $\hat{\alpha}=\,1+i\,\alpha +j\,\alpha^2+k\,\alpha^3, \, \, \hat{\beta }=\,1+i\,\beta +j\,\beta^2+k\,\beta^3$.\\
Using (\ref{G4}) and (\ref{G15}), the proof is easily seen.
\begin{equation*}
{\begin{array}{rl}
{QP}_{k,n}=&{P}_{k,n}+i\,{P}_{k,n+1}+\varepsilon\,{P}_{k,n+2}+i\,\varepsilon\,{P}_{k,n+3} \cr
\cr
=&\frac{\alpha^n -\beta^n}{\alpha -\beta}+i\,(\frac{\alpha^{n+1} -\beta^{n+1} }{\alpha -\beta})+\varepsilon\,(\frac{\alpha^{n+2} -\beta^{n+2}}{\alpha -\beta})+i\,\varepsilon\,(\frac{\alpha^{n+3} -\beta^{n+3} }{\alpha -\beta}) \cr
\cr
=&\frac{\alpha^{n}\,(1+i\,\alpha+\varepsilon\,\alpha^2+i\,\varepsilon\,\alpha^3)-\beta^{n}\,(1+i\,\beta+\varepsilon\,\beta^2+i\,\varepsilon\,\beta^3) }{\alpha -\beta} \cr
\cr
=&\frac{1}{2\,\sqrt{1+k}}\left(\hat{\alpha}\,{\alpha }^{n}-\hat{\beta}\,{\beta}^{n} \right).
\end{array}}
\end{equation*}
where $\hat{\alpha}=\,1+i\,\alpha +\varepsilon\,\alpha^2+i\,\varepsilon\,\alpha^3$, \, \, $\hat{\beta}=\,1+i\,\beta +\varepsilon\,\beta^2+i\,\varepsilon\,\beta^3$.
\end{proof}
\begin{thm} For $n,m\ge 0$ the d'Ocagne's identity for the  dual-complex k-Pell quaternions ${Q_P}_{k,n}$ and ${Q_P}_{k,m}$ is given by
\begin{equation}\label{G17}
\begin{array}{rl}
{Q_P}_{k,m}\,{Q_P}_{k,n+1}-{Q_P}_{k,m+1}\,{Q_P}_{k,n}=&(-1)^{n}\,k^n\,{P}_{k,m-n}\,[(1+k)+2\,i\, \cr
&+(2\,k^2+6\,k+4)\varepsilon+(4\,k+8)\,i\,\varepsilon\,].
\end{array}
\end{equation}
\end{thm}

\begin{proof}
(\ref{G13}): By using (\ref{G15}) we get,
\begin{equation*}
\begin{array}{rl}
{Q_P}_{k,m}\,{Q_P}_{k,n+1}-{Q_P}_{k,m+1}\,{Q_P}_{k,n}=&(\frac{\hat{\alpha}\,\alpha^m -\hat{\beta}\,\beta^m}{\alpha -\beta})(\frac{\hat{\alpha}\,\alpha^{n+1} -\hat{\beta}\,\beta^{n+1} }{\alpha -\beta}) \cr
&-(\frac{\hat{\alpha}\,\alpha^{m+1} -\hat{\beta}\,\beta^{m+1}}{\alpha -\beta})(\frac{\hat{\alpha}\,\alpha^n -\hat{\beta}\,\beta^n}{\alpha -\beta}) \cr
=&\frac{(\hat{\alpha}\,\hat{\beta})}{(\alpha -\beta)^2}\,(\alpha\beta)^n\,(\alpha^{m-n}-\beta^{m-n}) \cr
&(\alpha-\beta) \cr
=&\frac{(\hat{\alpha}\,\hat{\beta})}{(\alpha -\beta)}\,(\alpha\,\beta)^n\,(\alpha^{m-n}-\beta^{m-n}) \cr
=&(\hat{\alpha}\,\hat{\beta})(\alpha\,\beta)^n\,(\frac{\alpha^{m-n}-\beta^{m-n}}{\alpha-\beta}) \cr
=&(\hat{\alpha}\,\hat{\beta})(-k)^n\,{P}_{k,m-n} \cr
=&(-1)^{n}\,k^n\,{P}_{k,m-n}\,[\,(1+k)+2\,i \cr
&+(2\,k^2+6\,k+4)\varepsilon+(4\,k+8)\,i\,\varepsilon\,].
\end{array}
\end{equation*}
where \, $(\hat{\alpha}\,\hat{\beta})=(1-\alpha\,\beta)+i\,(\alpha+\beta)+\varepsilon(\alpha^2+\beta^2-\alpha\beta^3-\alpha^3\beta)+i\,\varepsilon(\alpha^3+\beta^3-\alpha\,\beta^2-\alpha^2\,\beta)=[\,(1+k)+2\,i+(2\,k^2+6\,k+4)\,\varepsilon+(4\,k+8)\,i\,\varepsilon\,]$. \\
\\
Calculate with a second method: By using (\ref{G4}) we get,
\begin{equation*}
{\begin{array}{rl}
{Q_P}_{k,m}\,{Q_P}_{k,n+1}-{Q_P}_{k,m+1}\,{Q_P}_{k,n}=&[\,({P}_{k,m}{P}_{k,n+1}-{P}_{k,m+1}{P}_{k,n}) \cr
&-({P}_{k,m+1}{P}_{k,n+2}-{P}_{m+2}{P}_{n+1})\,]\cr
&+\,i\,[\,{P}_{k,m}{P}_{k,n+2}-{P}_{k,m+2}{P}_{k,n}\,] \cr
&+\varepsilon\,[\,({P}_{k,m}{P}_{k,n+3}-{P}_{k,m+1}{P}_{k,n+2}) \cr
&-({P}_{k,m+1}{P}_{k,n+4}-{P}_{k,m+2}{P}_{k,n+3}) \cr
&+({P}_{k,m+2}{P}_{k,n+1}-{P}_{k,m+3}{P}_{k,n}) \cr
&-({P}_{k,m+3}{P}_{k,n+2}-{P}_{k,m+4}{P}_{k,n+1})\,] \cr
&+i\,\varepsilon\,[\,{P}_{k,m}{P}_{k,n+4}-{P}_{k,m+4}{P}_{k,n}\,] \cr
=&(-1)^n\,k^n\,(1+k)\,{P}_{m-n} \cr
&+2\,i\,(-1)^n\,k^n\,{P}_{m-n} \cr
&+\varepsilon\,[\,(-1)^n\,k^n\,(1+k)\, \cr
&(k^2\,{P}_{m-n-2}+{P}_{m-n+2})\,] \cr
&+i\,\varepsilon\,[\,(-1)^n\,k^n\,(4\,k+8)\,{P}_{k,m-n}\,] \cr
=&(-1)^{n}\,k^n\,{P}_{k,m-n}\,[\,(1+k)+2\,i\, \cr
&+(2\,k^2+6\,k+4)\varepsilon+(4\,k+8)\,i\,\varepsilon\,].
\end{array}}
\end{equation*}
where the identity ${P}_{k,m}{P}_{k,n+1}-{P}_{k,m+1}{P}_{k,n}=(-1)^n\,k^n\,{P}_{k,m-n}$ are used \, \cite{G}. Furthermore,
\begin{equation*}
{\begin{array}{l}
{P}_{k,m+2}\,{P}_{k,n+1}-{P}_{k,m+1}\,{P}_{k,n+2}=(-1)^{n}\,k^{n+1}\,{P}_{k,m-n},\cr
{P}_{k,m}\,{P}_{k,n+2}-{P}_{k,m+2}\,{P}_{k,n}=2\,(-1)^{n}\,k^n\,{P}_{k,m-n},\cr
{P}_{k,m}\,{P}_{k,n+3}-{P}_{k,m+1}\,{P}_{k,n+2}=(-1)^{n}\,k^{n+2}\,{P}_{k,m-n-2},\cr
{P}_{k,m+2}\,{P}_{k,n+1}-{P}_{k,m+3}\,{P}_{k,n}=(-1)^{n}\,k^n\,{P}_{k,m-n+2},\cr
{P}_{k,m+4}\,{P}_{k,n+1}-{P}_{k,m+3}\,{P}_{k,n+2}=(-1)^{n}\,k^{n+1}\,{P}_{k,m-n+2},\cr
{P}_{k,m+2}\,{P}_{k,n+3}-{P}_{k,m+1}\,{P}_{k,n+4}=(-1)^{n}\,k^{n+3}\,{P}_{k,m-n-2},\cr
{P}_{k,m}\,{P}_{k,n+4}-{P}_{k,m+4}\,{P}_{k,n}=(-1)^{n}\,k^{n}\,(8+4k)\,{P}_{k,m-n},\cr 
{P}_{k,m-n+2}+k^2\,{P}_{k,m-n-2}=(2\,k+4)\,{P}_{k,m-n}.
\end{array}}
\end{equation*}
are used. 
\end{proof}
\begin{thm} \textbf{Cassini's Identity} 
Let ${Q_P}_{k,n}$ be the dual-complex k-Pell quaternion. For $n\ge1$, Cassini's identity for ${Q_P}_{k,n}$ is as follows: 
\begin{equation}\label{G18}
\begin{aligned}
{Q_P}_{k,n-1}\,{Q_P}_{k,n+1}-{Q_P}_{k,n}^2=&(-1)^{n}\,k^{n-1}\,[(1+k)+2\,i\, \cr
&+(2\,k^2+6\,k+4)\varepsilon+(4\,k+8)\,i\,\varepsilon\,].
\end{aligned} 
\end{equation}
\end{thm}
\begin{proof}
(\ref{G18}): By using (\ref{G15}) we get,
\begin{equation*}
{\begin{array}{rl}
{Q_P}_{k,n-1}\,{Q_P}_{k,n+1}-{Q_P}_{k,n}^2=&(\frac{\hat{\alpha}\,\alpha^{n-1} -\hat{\beta}\,\beta^{n-1} }{\alpha -\beta})(\frac{\hat{\alpha}\,\alpha^{n+1} -\hat{\beta}\,\beta^{n+1} }{\alpha -\beta}) \cr
&-(\frac{\hat{\alpha}\,\alpha^{n}-\hat{\beta}\,\beta^{n}}{\alpha-\beta})^2 \cr
=&\frac{-(\hat{\alpha}\,\hat{\beta})}{(\alpha -\beta)^2}\,(\alpha\beta)^n\,(\alpha^{-1}\beta+\beta^{-1}\alpha-2) \cr
=&\frac{-1}{(\alpha -\beta)^2}\,(\hat{\alpha}\,\hat{\beta})\,(\alpha\beta)^n\,(\alpha\,\beta)(\frac{\alpha^{-1}\beta+\beta^{-1}\alpha-2}{\alpha\,\beta}) \cr
=&\frac{-1}{(\alpha -\beta)^2}\,(\hat{\alpha}\,\hat{\beta})\,(\alpha\beta)^n\,(\frac{\alpha^2+\beta^2}{\alpha\beta}-2) \cr
=&\frac{-1}{(\alpha -\beta)^2}\,(\hat{\alpha}\,\hat{\beta})\,(\alpha\beta)^n\,\frac{(\alpha-\beta)^2}{\alpha\beta} \cr
=&-(\hat{\alpha}\,\hat{\beta})(\alpha\beta)^{n-1} \cr
=&(-1)^{n}\,k^{n-1}\,[(1+k)+2\,i+(2\,k^2+6\,k+4)\varepsilon \cr
&+(4\,k+8)\,i\,\varepsilon\,]. 
\end{array}}
\end{equation*}
where \, $(\hat{\alpha}\,\hat{\beta})=(1-\alpha\,\beta)+i\,(\alpha+\beta)+\varepsilon(\alpha^2+\beta^2-\alpha\beta^3-\alpha^3\beta)+i\,\varepsilon(\alpha^3+\beta^3-\alpha\,\beta^2-\alpha^2\,\beta)=[\,(1+k)+2\,i+(2\,k^2+6\,k+4)\,\varepsilon+(4\,k+8)\,i\,\varepsilon\,]$. \\
\\
Calculate with a second method: By using (\ref{G4}) we get
\begin{equation*}
{\begin{array}{rl}
{Q_P}_{k,n-1}\,{Q_P}_{k,n+1}\,-({Q_P}_{k,n})^2=&({P}_{k,n-1}{P}_{k,n+1}-{P}_{n}^2)+({P}_{k,n+1}^2-{P}_{k,n}{P}_{k,n+2}) \cr
&-i\,({P}_{k,n+1}{P}_{k,n}-{P}_{k,n+2}{P}_{k,n-1}) \cr
&+\varepsilon\,[-({P}_{k,n+2}{P}_{k,n}-{P}_{k,n+3}{P}_{k,n-1}) \cr
&-({P}_{k,n}{P}_{k,n+2}-{P}_{k,n+1}{P}_{k,n+1}) \cr
&+({P}_{k,n+1}{P}_{k,n+3}-{P}_{k,n+2}{P}_{k,n+2})\cr
&+({P}_{k,n+3}{P}_{k,n+1}-{P}_{k,n+4}{P}_{k,n})] \cr
&-i\,\varepsilon\,({P}_{k,n+3}{P}_{k,n}-{P}_{k,n+4}{P}_{k,n-1})\cr
=&(-1)^{n}\,k^{n-1}\,[(1+k)+2\,i+(2\,k^2+6\,k+4)\varepsilon \cr
&+(4\,k+8)\,i\,\varepsilon\,]. 
\end{array}}
\end{equation*} 
where the identities of the k-Pell numbers ${{P}_{k,m}{P}_{k,n+1}-{P}_{k,m+1}{P}_{k,n}=(-1)^n\,k^n\,{P}_{k,m-n}}$\, and ${{P}_{k,n-1}{P}_{k,n+1}-{P}_{k,n}^2=(-1)^n\,k^{n-1}}$ are used \, \cite{G}. Furthermore,
\begin{equation*}
{\begin{array}{l}
{P}_{k,n-1}\,{P}_{k,n+2}-{P}_{k,n}\,{P}_{k,n+1}=2\,(-1)^{n}\,k^{n-1},\cr
{P}_{k,n-1}\,{P}_{k,n+3}-{P}_{k,n}\,{P}_{k,n+2}=(-1)^{n}\,k^{n-1}(4+k),\cr
{P}_{k,n+1}\,{P}_{k,n+3}-{P}_{k,n}\,{P}_{k,n+4}=(-1)^{n}\,k^n\,(4+k),\cr
{P}_{k,n+1}\,{P}_{k,n+1}-{P}_{k,n+2}\,{P}_{k,n}=(-1)^{n}\,k^n,\cr
{P}_{k,n+3}\,{P}_{k,n+1}-{P}_{k,n+2}\,{P}_{k,n+2}=(-1)^{n}\,k^{n+1},\cr
{P}_{k,n-1}\,{P}_{k,n+4}-{P}_{k,n}\,{P}_{k,n+3}=(-1)^{n}\,k^{n-1}\,(4\,k+8)\,. 
\end{array}}.
\end{equation*}
are used. 
\end{proof}

\begin{thm} \textbf{Catalan's Identity} 
Let ${Q_P}_{k,n}$ be the dual-complex k-Pell quaternion. For $n\ge 1$, Catalan's identity for ${Q_P}_{k,n}$ is as follows:
\begin{equation}\label{G19}
\begin{aligned}
{Q_P}_{k,n}^2-{Q_P}_{k,n+r}\,{Q_P}_{k,n-r}=&(-k)^{n-r+1}\,{P}_{k,r}^2\,[(1+k)+2\,i+(2\,k^2+6\,k+4)\varepsilon \cr
&+(4\,k+8)\,i\,\varepsilon].
\end{aligned}
\end{equation}
\end{thm}
\begin{proof}
(\ref{G19}): By using (\ref{G15}) we get 
\begin{equation*}
\begin{array}{rl}
{Q_P}_{k,n-r}\,{Q_P}_{k,n+r}-{Q_P}_{k,n}^2=&(\frac{\hat{\alpha}\,\alpha^{n-r} -\hat{\beta}\,\beta^{n-r} }{\alpha -\beta})(\frac{\hat{\alpha}\,\alpha^{n+r} -\hat{\beta}\,\beta^{n+r} }{\alpha -\beta})-(\frac{\hat{\alpha}\,\alpha^{n} -\hat{\beta}\,\beta^{n} }{\alpha -\beta})^2 \cr
=&\frac{-(\hat{\alpha}\,\hat{\beta})}{(\alpha -\beta)^2}\,(\alpha\beta)^n\,[(\alpha^{-r}\beta^r+\beta^{-r}\alpha^r-2)\, \cr
=&\frac{1}{(\alpha -\beta)^2}\,[-(\hat{\alpha}\,\hat{\beta})\,(\alpha\beta)^n\,(\alpha\,\beta)^r\,(\frac{\alpha^{-r}\beta^r+\beta^{-r}\alpha^r-2}{\alpha\,\beta}^r) \cr
=&\frac{1}{(\alpha -\beta)^2}\,[-(\hat{\alpha}\,\hat{\beta})\,(\alpha\beta)^n\,(\frac{\alpha^{2r}+\beta^{2r}}{{\alpha\beta}^r}-2)] \cr
=&\frac{1}{(\alpha -\beta)^2}\,[-(\hat{\alpha}\,\hat{\beta})\,(\alpha\beta)^{n-r}\,(\alpha^r-\beta^r)^2] \cr
=&-(\hat{\alpha}\,\hat{\beta})(\alpha\beta)^{n-r}\,(\frac{\alpha^r-\beta^r}{\alpha -\beta})^2 \cr
=&(-1)^{n-r+1}\,k^{n-r}\,{P}_{k,r}^2\,[\,(1+k)+2\,i \cr
&+(2\,k^2+6\,k+4)\,\varepsilon+(4\,k+8)\,i\,\varepsilon\,]
\end{array}
\end{equation*}
where \, $(\hat{\alpha}\,\hat{\beta})=(1-\alpha\,\beta)+i\,(\alpha+\beta)+\varepsilon(\alpha^2+\beta^2-\alpha\beta^3-\alpha^3\beta)+i\,\varepsilon(\alpha^3+\beta^3-\alpha\,\beta^2-\alpha^2\,\beta)=(1+k)+2\,i+(2\,k^2+6\,k+4)\,\varepsilon+(4\,k+8)\,i\,\varepsilon$.  
\end{proof}

\section{Conclusion} 
In this study, a number of new results on dual-complex k-Pell quaternions were derived. Quaternions have great importance as they are used in quantum physics, applied mathematics, quantum mechanics, Lie groups, kinematics and differential equations. \\ 
This study fills the gap in the literature by providing the dual-complex k-Pell quaternion using definitions of the dual-complex number \cite{V} and k-Pell number \cite{F}.

\end{document}